\documentclass[preprint,12pt]{elsarticle}

\usepackage{amsthm}
\usepackage{amsmath,amssymb,txfonts}
\usepackage{amsthm,color,bm,comment}

\newtheorem{theorem}{Theorem}
\newtheorem{lemma}{Lemma}
\newtheorem{prop}{Proposition}

\newtheorem{remark}{Remark}

\theoremstyle{definition}

\def\re{\mathbb{R}}
\def\N{\mathbb{N}}
\def\Z{\mathbb{Z}}

\def\({\left(}
\def\){\right)}
\def\[{\left[}
\def\]{\right]}
\def\pd{\partial}
\def\lap{\Delta}

\def\w{\omega}
\def\la{\lambda}

\begin{document}

\begin{frontmatter}



\title{Remarks on a limiting case of Hardy type inequalities}


\author[M]{Megumi Sano\corref{Sano}\fnref{label1}}
\ead{smegumi@hiroshima-u.ac.jp}
\fntext[label1]{Corresponding author.}
\address[M]{Laboratory of Mathematics, Graduate School of Engineering, Hiroshima University, Higashi-Hiroshima, 739-8527, Japan}

\author[T]{Takuya Sobukawa\corref{Sobu}}
\ead{sobu@waseda.jp}
\address[T]{Global Education Center, Waseda University, Nishi-Waseda, Shinjuku-ku, Tokyo 169-8050, Japan}


\date{\today}

\begin{abstract}
The classical Hardy inequality holds in Sobolev spaces $W_0^{1,p}$ when $1\le p< N$. 
In the limiting case where $p=N$, it is known that by adding a logarithmic function to the Hardy potential, some inequality which is called the critical Hardy inequality holds in $W_0^{1,N}$. 
In this note, in order to give an explanation of appearance of the logarithmic function at the potential, we derive the logarithmic function from the classical Hardy inequality with the best constant via some limiting procedure as $p \nearrow N$.  
And we show that our limiting procedure is also available for the classical Rellich inequality in second order Sobolev spaces $W_0^{2,p}$ with $p \in (1, \frac{N}{2})$ and the Poincar\'e inequality.
\end{abstract}

\end{frontmatter}



%
%
\section{Introduction}

Let $B_1 \subset \re^N$ be the unit ball, $1 < p <N$ and $N \ge 2$. The classical Hardy inequality
\begin{equation}
\label{H_p}
\( \frac{N-p}{p} \)^p \int_{B_1} \frac{|u|^p}{|x|^p} dx \le \int_{B_1} | \nabla u |^p dx,
\end{equation}
holds for all $u \in W^{1,p}_0(B_1)$, where $W_0^{1,p}(B_1)$ is a completion of $C_c^{\infty}(B_1)$ with respect to the norm $\| \nabla (\cdot )\|_{L^p(B_1)}$. 
Note that the inequality (\ref{H_p}) expresses the embedding $W_0^{1,p}(B_1) \hookrightarrow L^p(B_1 ; |x|^{-p} dx)$ which is equivalent to $W_0^{1,p}(B_1) \hookrightarrow L^{p^*, p}(B_1)$ thanks to the rearrangement technique, where $p^* = \frac{Np}{N-p}$ and $L^{p,q}$ is the Lorentz space. Therefore by a property of Lorentz spaces, we see that for any $q > p$
\begin{align*}
W_0^{1,p} \hookrightarrow L^{p^*, p} \hookrightarrow L^{p^*, q} \hookrightarrow L^{p^*, \infty}.
\end{align*}
The variational problems and partial differential equations associated with the best constants of the inequalities related to the embeddings as above are classical interesting topics, see \cite{T, L, BG, BV, VZ}, to name a few.

On the other hand, in the limiting case where $p=N$ the classical Hardy inequality (\ref{H_p}) does not have its meaning due to the best constant becomes zero. Instead of it, by adding a logarithmic function at the Hardy potential, the following inequality which is called the critical Hardy inequality:
\begin{equation}
\label{H_N}
\( \frac{N-1}{N} \)^N \int_{B_1} \frac{|u|^N}{|x|^N (\log \frac{a}{|x|})^N} dx \le \int_{B_1} | \nabla u |^N dx
\end{equation}
holds for all $u \in W^{1,N}_0(B_1)$ and $a \ge 1$, see \cite{Leray, Ladyzhenskaya}.
The inequality (\ref{H_N}) expresses the embedding: $W_0^{1,N}(B_R) \hookrightarrow L^N(B_1 ; |x|^{-N}(\log \frac{a}{|x|})^{-N} dx)$. Since for large $a$ the weight functions $|x|^{-N}(\log \frac{a}{|x|})^{-N}$ are radially decreasing with respect to $|x|$, the embedding with $a >1$ is equivalent to $W_0^{1,N} \hookrightarrow L^{\infty, N}(\log L)^{-1}$ thanks to the rearrangement technique. Here $L^{p,q}(\log L)^r$ is the Lorentz-Zygmund space which is given by
\begin{align*}
L^{p,q}(\log L)^r &= \left\{ \, u: B_1 \to \re\,\,{\rm measurable} \,\,\biggr| \,\, \| u\|_{L^{p,q} (\log L)^r} < \infty  \, \right\} \\
\| u \|_{L^{p,q}(\log L)^r} &= 
\begin{cases}
\( \int_0^{|B_1|} s^{\frac{q}{p}-1} \( 1+ \log \frac{|B_1| }{s} \)^{rq} (u^* (s) )^q \,ds \)^{\frac{1}{q}} \quad &\text{if} \,\, 1\le q <\infty, \\
\sup_{0< s < |B_1|} s^{\frac{1}{p}} \( 1+ \log \frac{|B_1|}{s} \)^r u^* (s)   &\text{if} \,\, q = \infty,
\end{cases}
\end{align*}
where $u^*$ be the Schwartz symmetrization of $u$. 
Note that $L^{p,q}(\log L)^0$ becomes the Lorentz space $L^{p,q}$ and $L^{\infty, \infty}(\log L)^r$ becomes the Zygmund space $Z^{-r}$ which can be  equivalent to Orlicz space $L_{e^{\,|u|^{-1/r}}} = {\rm ExpL}^{-\frac{1}{r}}$ in some sense (see \cite{BR, BS, CRT(2013)}). 
By a property of Lorentz-Zygmund spaces (see e.g. \cite{BS} Theorem 9.5.), we see that for any $q >N$ 
\begin{align*}
W_0^{1,N} \hookrightarrow L^{\infty, N}(\log L)^{-1} \hookrightarrow L^{\infty, q}(\log L)^{-1+ \frac{1}{N} - \frac{1}{q}} \hookrightarrow L^{\infty, \infty}(\log L)^{-1+\frac{1}{N}} = {\rm Exp L}^{\frac{N}{N-1}}.
\end{align*}
And variational problems associated with the best constants of the inequalities related to embeddings as above are also studied, see \cite{A, CC, Adimurthi-Sandeep, HK, CRT(2013), II, S}.


In this note, in order to give an explanation of appearrance of the logarithmic function in the limiting case $p=N$ of the classical Hardy inequality, we shall derive the logarithmic function in (\ref{H_N}) from the classical Hardy inequality with the best constant by some limiting procedure as $p \nearrow N$ based on extrapolation. Giving an explanation of it is corresponding to giving a reason {\it why we consider Lorentz-Zygmund space $L^{p,q}(\log L)^r$ in the embedding of the critical Sobolev space $W_0^{1,N}$}. In this viewpoint, considering a limiting procedure for the classical Hardy inequality is significant. Main Theorem is as follows.

\begin{theorem}\label{H_N from H_p}
The following non-sharp critical Hardy inequality (\ref{H_N non-sharp}) can be derived by a limiting procedure for the classical Hardy inequality (\ref{H_p}) as $p \nearrow N$. 
\begin{align}\label{H_N non-sharp}
C \int_{B_1} \frac{|u|^N}{|x|^N \( \log \frac{a}{|x|} \)^\beta} \,dx \le \int_{B_1} \left| \nabla u \cdot \frac{x}{|x|} \right|^N \,dx \quad (u \in C_c^1(B_1)).
\end{align}
Here $\beta > 2N, a>1$, and the constant $C= C(\beta, a, N)>0$ is independent of $u$.
\end{theorem}

Note that the inequality (\ref{H_N non-sharp}) does not have the optimal exponent $\beta$ and the optimal constant $C$, and itself is already well-known. 
However, our limiting procedure for the classical Hardy inequality is new and giving some explanation of appearence of the logarithmic function at the Hardy potential in the limiting case $p=N$. 
Our limiting procedure can be regarded as an analogue of Trudinger's argument in \cite{Tru} for the Sobolev inequality as $p \nearrow N$, see also \cite{CR} Theorem 1.7..
For several limiting procedures, we refer \cite{BP} (The Sobolev inequality as $N \nearrow \infty$), \cite{Y}, \cite{Z}XII 4.41. ($L^p$ boundedness of the Hilbert transformation as $p \searrow 1$ or $p\nearrow \infty$), see also \cite{SC} Corollary 3.2.4 (The Sobolev inequality is derived from the Nash inequality).


A few comments on Theorem \ref{H_N from H_p} are in order. Very recently, Ioku \cite{I} showed the following improved Hardy inequality (\ref{iH_p}) on $B_1$ which is equivalently connected to the classical Hardy inequality on $\re^N$ via the following transformation (\ref{itrans}) at the radial setting.
\begin{align}\label{itrans}
&w( |y| ) = u (|x|), \,\,\text{where}\,\,|y|^{-\frac{N-p}{p-1}} = |x|^{-\frac{N-p}{p-1}} -1\\
&\int_{\re^N} | \nabla w |^p dy - \( \frac{N-p}{p} \)^p \int_{\re^N} \frac{|w|^p}{|y|^p} dy \notag \\
\label{iH_p}
&=\int_{B_1} | \nabla u |^p dx - \( \frac{N-p}{p} \)^p \int_{B_1} \frac{|u|^p}{|x|^p \( 1- |x|^{\frac{N-p}{p-1}} \)^p} dx \ge 0.
\end{align}
And also, Ioku showed the improved Hardy inequality (\ref{iH_p}) yields the critical Hardy inequality (\ref{H_N}) with the best constant by taking the limit $p \nearrow N$. 
However, in the higher order or fractional order case, these beautiful and simple  structures and transformations might be nothing. 
In this note, we also treat the second order case. 
We observe that the transformation (\ref{itrans}) connets two singular solutions of $p-$Laplace equations: $-{\rm div}\,(|\nabla u|^{p-2} \nabla u ) = 0$ on $B_1$ and $\re^N$. On the other hand, by considering a transformation which connects two singular solutions of $p-$Laplace equation and $N-$Laplace equation, the authors \cite{ST} showed an equivalence between the classical Hardy inequality (\ref{H_p}) and the critical Hardy inequality (\ref{H_N}), for the Sobolev type inequalities, see \cite{Za, S}. Furthermore, for other transformations and another interesting equivalence, see \cite{HK, II(2016), CRT(2018)}.


This paper is organized as follows:
In section \ref{Pre}, necessary preliminary facts are presented. 
In section \ref{Hardy}, we give the limiting procedure as $p \nearrow N$ for the Hardy inequality in Theorem \ref{H_N from H_p}. We also apply our limiting procedure for the Rellich inequality. 
In section \ref{Poincare}, we consider {\it a limit} as $|\Omega| \searrow 0$ for the Poincar\'e inequality via our limiting procedure.

We fix several notations: 
$B_R$ denote a $N$-dimensional ball centered $0$ with radius $R$ and $\omega_{N-1}$ denotes an area of the unit sphere in $\re^N$. $|A|$ denotes the Lebesgue measure of a set $A \subset \re^N$ and $X_{{\rm rad}} = \{ \, u \in X \, | \, u \,\,\text{is radial} \, \}.$
Throughout the paper, if a radial function $u$ is written as $u(x) = \tilde{u}(|x|)$ by some function $\tilde{u} = \tilde{u}(r)$, we write $u(x)= u(|x|)$ with admitting some ambiguity.

%
%

\section{Preliminaries}\label{Pre}


In this section, we prepare several lemmas to show Theorems. 
The following pointwise estimate is well-known. we omit the proof.

\begin{lemma}\label{radial lemma}
For any radial functions $u \in C^1 (B_R) \cap C(\overline{B_R})$ satisfying $u|_{\pd B_R} =0$, for any $r \in (0, R)$ the following estimate holds.
\begin{align*}
|u(r)| &\le 
\begin{cases}
\vspace{0.5em}
\w_{N-1}^{-1} \| \nabla u\|_{L^p(B_R)} r^{-(N-1)}  &\text{if} \,\,\, p =1, \\
\vspace{0.5em}
\( \frac{p-1}{|\,N-p \,|} \)^{\frac{p-1}{p}} \w_{N-1}^{-\frac{1}{p}}  \| \nabla u\|_{L^p(B_R)} \left| \, r^{-\frac{N-p}{p-1}} -  R^{-\frac{N-p}{p-1}} \,\right|^{\frac{p-1}{p}} \quad &\text{if} \,\,\, 1 < p \not\eq N, \\
\w_{N-1}^{-\frac{1}{N}} \| \nabla u \|_{L^N(B_1 )} \( \log \frac{R}{r} \)^{\frac{N-1}{N}} &\text{if} \,\,\, p = N.
\end{cases}
\end{align*}
\end{lemma}


%
%

When the potential function is not radially decreasing, we can not apply rearrangement technique. Instead of rearrangement, we prepare the following lemma which can reduces a problem to the radial setting.

\begin{lemma}\label{radial nomi}
Let $1 < q <\infty$, $V=V(x)$ be a radial function on $B_R$. 
If there exists $C>0$ such that for any radial functions $u \in C_c^1(B_R)$ the inequality
\begin{equation}\label{rad}
C \int_{B_R} |u|^q V(x)\, dx \le \int_{B_R} | \nabla u |^q\, dx < \infty
\end{equation}
holds, then for any functions $w \in C_c^1(B_R)$ the inequality 
\begin{equation}\label{non-rad}
C \int_{B_R} |w|^q V(x)\, dx \le \int_{B_R} \left| \nabla w \cdot \frac{x}{|x|} \right|^q\, dx < \infty
\end{equation}
holds.
\end{lemma}

\begin{proof}
For any $w \in C_c^1(B_R)$, define a radial function $W$ as follows.
\begin{align*}
W(r) = \( \w_{N-1}^{-1} \int_{\w \in \pd B_1} | w(r\w ) |^q\,dS_{\w} \)^{\frac{1}{q}} \quad (0 \le r \le R).
\end{align*}
Then we have
\begin{align*}
|W \,'(r) | &= \w_{N-1}^{-\frac{1}{q}} \( \int_{\pd B_1} | w(r\w ) |^q\,dS_{\w} \)^{\frac{1}{q}-1} \int_{\pd B_1} | w |^{q-1} \left| \frac{\pd w}{\pd r} \right| \,dS_{\w} \\
&\le \w_{N-1}^{-\frac{1}{q}} \( \int_{\pd B_1} \left| \frac{\pd w}{\pd r} (r\w ) \right|^q \,dS_{\w} \)^{\frac{1}{q}}.
\end{align*}
Therefore we have
\begin{align}\label{right}
\int_{B_R} | \nabla W |^q\, dx &\le \int_{B_R} \left| \nabla w \cdot \frac{x}{|x|} \right|^q\, dx, \\
\label{left}
\int_{B_R} |W|^q V(x)\, dx &= \int_{B_R} |w|^q V(x)\, dx.
\end{align}
From (\ref{rad}) for $W$, (\ref{right}), and (\ref{left}), we obtain (\ref{non-rad}) for any $w$.
\end{proof}

We show the pointwise estimates for radial functions and their derivative in $W_0^{2, p}(B_R)$ when $N \ge 3$ and $p \ge 1$. 
When $p=1$ or $N =2$, the pointwise estimates are already shown by \cite{CRT(2009), CRT(2010)}. 

\begin{lemma}\label{2nd-order radial lemma}
Let $N \ge 3$ and $u \in C^2(B_R) \cup C(\overline{B_R})$ be a radial function satisfying $u|_{\pd B_R} =0$. Then the following pointwise estimates hold for any $r \in (0, R)$.
\begin{align}\label{estimate u}
|u(r)| &\le 
\begin{cases}
\vspace{0.5em}
\frac{p}{|\,N-2p \,|} \w_{N-1}^{-\frac{1}{p}} N^{\frac{1}{p}-1}  \| \lap u\|_{L^p(B_R)} \left| \, r^{-\frac{N-2p}{p}} -  R^{-\frac{N-2p}{p}} \,\right| \quad &\text{if} \,\,\, 1 \le p \not\eq \frac{N}{2}, \\
\vspace{0.2em}
\w_{N-1}^{-\frac{2}{N}} N^{\frac{2}{N}-1} \| \lap u\|_{L^{N/2}(B_R)} \log \frac{R}{r} &\text{if} \,\,\, p = \frac{N}{2}.
\end{cases} \\ 
\notag \\
\label{estimate u'}
|u'(r)| &\le \frac{\| \lap u\|_{L^p(B_R)}}{\w_{N-1}^{\frac{1}{p}} N^{1-\frac{1}{p}}} r^{-\frac{N-p}{p}} \quad \text{for any} \,p \ge 1.
\end{align}
\end{lemma}

\begin{proof}
Consider the following transformation (ref. \cite{CRT(2010)}):
\begin{align}\label{trans}
w(t) = A u(r), \,\,\text{where}\,\, r= R (t +1)^{-\frac{1}{N-2}}\,\,\text{and}\,\,A^p = \w_{N-1} R^{N-2p} (N-2)^{2p -1}.
\end{align}
Then we have
\begin{align}\label{lap to ''}
w''(t) = \frac{A R^2}{(N-2)^2} (t+1)^{-2\frac{N-1}{N-2}} \( u''(r) + \frac{N-1}{r} u'(r) \) 
= \frac{A R^2}{(N-2)^2} (t+1)^{-2\frac{N-1}{N-2}} \lap u 
\end{align}
which yields that
\begin{align*}
\int_{B_R} |\lap u|^p \,dx = \int_0^{\infty} | w''(t) |^p (t+1)^{2\frac{(N-1)(p-1)}{N-2}} \,dt.
\end{align*}
Since $w(0)= w'(\infty)=0$, we have
\begin{align*}
w(t) &= - \int_0^t \int_s^\infty w''(u) \,du \,ds \\
&\le \int_0^t \( \int_0^\infty |w''(u)|^p (u+1)^{2\frac{(N-1)(p-1)}{N-2}} \,du \)^{\frac{1}{p}} \( \int_s^\infty (u+1)^{-2\frac{N-1}{N-2}}\,du \)^{\frac{p-1}{p}} \,ds \\
&= \( \frac{N-2}{N} \)^{\frac{p-1}{p}} \| \lap u\|_{L^p(B_R)} \int_0^t (s+1)^{-\frac{N(p-1)}{(N-2)p}} \,ds \\
&= 
\begin{cases}
\( \frac{N-2}{N} \)^{\frac{p-1}{p}} \frac{(N-2)p}{N-2p} \| \lap u\|_{L^p(B_R)} \( (t+1)^{\frac{N-2p}{(N-2)p}} -1\)  \quad &\text{if} \,\, p \not= \frac{N}{2}, \\
\( \frac{N-2}{N} \)^{1-\frac{2}{N}}  \| \lap u\|_{L^{N/2}(B_R)} \log (t+1) &\text{if} \,\, p = \frac{N}{2}.
\end{cases}
\end{align*}
Therefore we obtain (\ref{estimate u}). On the other hand, since
\begin{align*}
w'(t) = - \int_t^\infty w''(u) \,du 
\le \( \frac{N-2}{N} \)^{\frac{p-1}{p}} \| \lap u\|_{L^p(B_R)} (t+1)^{\frac{N(p-1)}{(N-2)p}}
\end{align*}
and $w'(t) = - A u'(r) \frac{R}{N-2} (t+1)^{-\frac{N-1}{N-2}}$, we also obtain (\ref{estimate u'}).
\end{proof}

For much higer order case, see Proposition \ref{higher order radial lemma} in \S \ref{Hardy}.

%
%

\section{A limiting procedure for the Hardy type inequalities}\label{Hardy}

\subsection{Proof of Theorem \ref{H_N from H_p}: The Hardy inequality}\label{limiting procedure}

First, we prepare for making the optimal constant $(\frac{N-p}{p})^p (\searrow 0$ as $p \nearrow N$) compete with $\int_{B_1} \frac{|u|^p}{|x|^p} \,dx \,(\nearrow \infty$ as $p \nearrow N$, in general). 

Let $p_k=N-\frac{1}{k}$ for $k \in \N$, $f \in C^1 (-\infty, \infty)$ be a  monotone-decreasing function which satisfies $\lim_{t \to +\infty} f(t) =0$, and $\{ \phi_k \}_{k \in \Z} \subset  C_c^{\infty}(\re^N \setminus \{ 0\})$ be radial functions which satisfy
\begin{align*}
&(i) \,\sum_{k=-\infty}^{+\infty} \phi_k (x)^N =1, \,0 \le \phi_k (x) \le 1 \,\, \( \forall x \in \re^N \setminus \{ 0\} \), \\
&(ii) \,\,\text{supp}\,\phi_k \subset B_{f(k)} \setminus B_{f(k+2)}.
\end{align*}
For any radial function $u \in C_c^1 (B_1)$, set $u_k = u \,\phi_k$ and $A_k =$ supp\,$u_k \subset B_1 \cap \( B_{f(k)} \setminus B_{f(k+2)} \)$. 
In order to obtain {\it a limit} for the classical Hardy inequality (\ref{H_p}) as $p \nearrow N$, the left-hand side of (\ref{H_p}) for $u_k$ and $p_k$ must not be vanishing as $k \to \infty$. We shall determine such $f$. Note that if $x \in A_k$, then $f(k+2) \le |x| \le f(k)$ and $k \le f^{-1}(|x|) \le k+2$. By Lemma \ref{radial lemma} we have
\begin{align}\label{H_p_k}
&\( \frac{N-p_k}{p_k} \)^{p_k} \int_{A_k} \frac{|u_k|^{p_k}}{|x|^{p_k}} dx 
= p_k^{-p_k} \int_{A_k} \( \, \frac{|u_k (x)|}{|x| \,k} \, \)^{N-\frac{1}{k}} dx \notag \\
&\ge C \int_{A_k} \frac{|u_k (x)|^N}{|x|^N \( f^{-1}(|x|) \)^N} \( \, \frac{|x|\, k}{|u_k (x)|} \, \)^{\frac{1}{k}} dx \notag \\
&\ge C \, \| \nabla u_k \|_{L^N(A_k)}^{-\frac{1}{k}} \int_{A_k} \frac{|u_k (x)|^N}{|x|^N \( f^{-1}(|x|) \)^N} \(  f(k+2) \( \log \frac{f(k)}{f(k+2)} \)^{-\frac{N-1}{N}} \, \)^{\frac{1}{k}} dx.
\end{align}
Therefore, if for any $k \in \N$ the function $f$ satisfies
\begin{align}\label{f}
\(  f(k+2) \( \log \frac{f(k)}{f(k+2)} \)^{-\frac{N-1}{N}} \, \)^{\frac{1}{k}} \ge C >0,
\end{align}
then the information on the left-hand side of the classical Hardy inequality (\ref{H_p}) is not vanishing in this limiting procedure. From (\ref{f}) and l'H\^opital's rule, we have an ordinary differential inequality for $f$ as follows:
\begin{align*}
\frac{d}{dt} f (t) \ge -C f (t)
\end{align*}
whose solution satisfies $f(t) \ge e^{-Ct}$. Thus $f^{-1}(t) \ge \frac{1}{C} \log \frac{1}{t}$. We belive that the above caluculation and consideration give some explanation of appearance of the logarithmic function at the Hardy potential in the limiting case $p=N$.

Hereinafter we set $f(t)= e^{-t}$.


\begin{proof}[{\bf Proof of Theorem \ref{H_N from H_p}}] 
From Lemma \ref{radial nomi}, it is enough to show the inequality (\ref{H_N non-sharp}) for any radial functions $u \in C_c^{1}(B_1)$. 
Applying the classical Hardy inequality (\ref{H_p}) for $u_k$ and $p_k$ for $k \ge 1$, we have 
\begin{align*}
\( \frac{N-p_k}{p_k} \)^{p_k} \int_{A_k} \frac{|u_k|^{p_k}}{|x|^{p_k}} dx \le \int_{A_k} | \nabla u_k |^{p_k} dx \le |A_k |^{1-\frac{p_k}{N}} \| \nabla u_k \|_N^{N-\frac{1}{k}}.
\end{align*}
By (\ref{H_p_k}) and (\ref{f}), for $k \ge 1$
\begin{align*}
C \int_{A_k} \frac{|u_k|^N}{|x|^N \( \log \frac{1}{|x|} \)^N} \,dx \le \int_{A_k} | \nabla u_k |^N dx.
\end{align*}
Since $k \le \log \frac{1}{|x|}$ for $x \in A_k$, 
\begin{align}\label{H_N_k}
C \int_{A_k} \frac{|u_k|^N}{|x|^N \( \log \frac{a}{|x|} \)^\beta} \,dx \le b_k \int_{A_k} | \nabla u_k |^N dx
\end{align}
for $k \ge 1, a>1$, and $\beta > 2N$, where $b_k$ is given by
\begin{align*}
b_k=
\begin{cases}
k^{N-\beta} \quad &\text{if} \,\,\,k\ge 1,\\
1 &\text{if} \,\,\,k=0, -1,\\
0 &\text{if} \,\,\,k\le -2.
\end{cases}
\end{align*}
Here, note that the inequalities (\ref{H_N_k}) with $k=0, -1$ come form the Poincar\'e inequality and the boundedness of the function $|x|^{-N}(\log \frac{a}{|x|})^{-\beta}$ on $A_0 \cup A_{-1} \subset B_1 \setminus B_{e^{-2}}$.
Summing both sides on (\ref{H_N_k}), we have
\begin{align*}
C \sum_{k \in \Z} \int_{B_1} \frac{|u \phi_k|^N}{|x|^N \( \log \frac{a}{|x|} \)^\beta} \,dx \le \sum_{k \in \Z} b_k \int_{A_k} | \nabla (u \phi_k ) |^N dx
\end{align*}
which yields that
\begin{align}\label{H_N_k sum}
C \int_{B_1} \frac{|u |^N}{|x|^N \( \log \frac{a}{|x|} \)^\beta} \,dx &\le 2^{N-1} \sum_{k \in \Z} b_k \int_{A_k} \phi_k^N | \nabla u |^N + |u|^N |\nabla \phi_k |^N dx \notag \\
&\le 2^{N-1} \int_{B_1} | \nabla u |^N dx + C \sum_{k=1}^{+\infty} b_k e^{kN} \int_{A_k} |u|^N \, dx.
\end{align}
By Lemma \ref{radial lemma} we have
\begin{align*}
b_k e^{kN} \int_{A_k} |u|^N \, dx 
&\le C b_k e^{kN} \| \nabla u \|_N^N \int_{A_k} \( \log \frac{1}{|x|} \)^{N-1} \, dx \\
&\le C b_k e^{kN} \| \nabla u \|_N^N \int_k^{k+2} s^{N-1} e^{-sN} \, ds \le C b_k k^{N-1} \| \nabla u \|_N^N.
\end{align*}
From (\ref{H_N_k sum}) we have
\begin{align*}
C \int_{B_1} \frac{|u |^N}{|x|^N \( \log \frac{a}{|x|} \)^\beta} \,dx 
&\le C \int_{B_1} | \nabla u |^N dx + C \( \sum_{k=1}^{+\infty} k^{-1-(\beta -2N)} \) \int_{B_1} |\nabla u|^N \, dx \\
&\le C \int_{B_1} |\nabla u|^N \, dx.
\end{align*}
\end{proof}

\subsection{The Rellich inequality}\label{Rellich}

Let $1< p< \frac{N}{2}$. The classical Rellich inequality:
\begin{equation}
\label{R_p}
\( \frac{N(p-1)(N-2p)}{p} \)^p \int_{B_1} \frac{|u|^p}{|x|^{2p}} dx \le \int_{B_1} | \lap u |^p dx
\end{equation}
holds for all $u \in W^{2,p}_0(B_1)$, where $W_0^{2,p}(B_1)$ is a completion of $C_c^{\infty}(B_1)$ with respect to the norm $\| \lap (\cdot )\|_{L^p(B_1)}$ (see \cite{Rellich}, \cite{Davies-Hinz}, \cite{Mitidieri}). 
In this section, we apply our limiting procedure in \S \ref{limiting procedure} to the Rellich inequality (\ref{R_p}) as $p \nearrow \frac{N}{2}$.

\begin{theorem}\label{R_N from R_p}
The following non-sharp critical Rellich inequality (\ref{R_N non-sharp}) can be derived by a limiting procedure for the classical Rellich inequality (\ref{R_p}) as $p \nearrow \frac{N}{2}$. 
\begin{align}\label{R_N non-sharp}
C \int_{B_1} \frac{|u|^{\frac{N}{2}}}{|x|^N \( \log \frac{a}{|x|} \)^\beta} \,dx \le \int_{B_1} | \lap u |^{\frac{N}{2}} \,dx  \quad (u \in C_{c, \,\text{rad}}^2(B_1)).
\end{align}
Here $\beta > N+2, a>1$, and the constant $C= C(\beta, a, N)>0$ is independent of $u$.
\end{theorem}

\begin{remark}
In the limiting case $p = \frac{N}{2}$, the inequality (\ref{R_N non-sharp}) with the optimal exponent $\beta$ and its best constant is already known, see \cite{DHA}, \cite{AS}. 
\end{remark}




\begin{proof}
We shall show (\ref{R_N non-sharp}) for any $u \in C_{c, \, \text{rad}}^2 (B_1)$. The strategy of the proof is the same as it in \S \ref{limiting procedure}. 

Let $p_k = \frac{N}{2}-\frac{1}{2k}$ for $k \ge 2$ and only condition (i) of $\phi_k$ in \S \ref{limiting procedure} is changed to $\sum_{k=-\infty}^{+\infty} \phi_k (x)^{\frac{N}{2}} =1$ for any $x \in \re^N \setminus \{ 0 \}$.
Applying the classical Rellich inequality (\ref{R_p}) for $u_k = u\,\phi_k$ for $u \in C_{c, \, \text{rad}}^2 (B_1)$ and $p_k$ for $k \ge 2$, we have 
\begin{align}\label{R_p_k}
\( \frac{(N-2 p_k)(p_k -1) N}{p_k} \)^{p_k} \int_{A_k} \frac{|u_k|^{p_k}}{|x|^{2 p_k}} dx \le \int_{A_k} | \lap u_k |^{p_k} dx.
\end{align}
On the left-hand side of (\ref{R_p_k}), by (\ref{estimate u}) in Lemma \ref{2nd-order radial lemma} we  have
\begin{align*}
&\( \frac{(N-2 p_k)(p_k -1) N}{p_k} \)^{p_k} \int_{A_k} \frac{|u_k|^{p_k}}{|x|^{2 p_k}} dx 
\ge C \int_{A_k} \( \, \frac{|u_k (x)|}{|x|^2 k} \, \)^{\frac{N}{2} - \frac{1}{2k}} dx \\
&\ge C \int_{A_k} \frac{|u_k (x)|^{\frac{N}{2}}}{|x|^N \( f^{-1}(|x|) \)^{\frac{N}{2}}} \( \, \frac{|x|^2 k}{|u_k (x)|} \, \)^{\frac{1}{2k}} dx \\
&= C \| \lap u_k \|_{L^{\frac{N}{2}}(A_k)}^{-\frac{1}{2k}} \int_{A_k} \frac{|u_k (x)|^{\frac{N}{2}}}{|x|^N \( f^{-1}(|x|) \)^{\frac{N}{2}}} \( \, f(k+2)^2 \(  \log \frac{f(k)}{f(k+2)} \)^{-1} \, \)^{\frac{1}{2k}} dx.
\end{align*}
If we choose $f(t)= e^{-t}$, then the left-hand side of (\ref{R_p_k}) is not vanishing as $k \to \infty$. Thus we set $f(t)= e^{-t}$ hereinafter.
In the similar way to it in \S \ref{limiting procedure}, for $a>1, k \in \Z$, and $\beta > N+2$ we have
\begin{align}\label{R_N_k}
C \int_{A_k} \frac{|u_k|^\frac{N}{2}}{|x|^N \( \log \frac{a}{|x|} \)^\beta} \,dx \le b_k \int_{A_k} | \lap u_k |^{\frac{N}{2}} dx,
\end{align}
where $b_k$ is given by
\begin{align*}
b_k=
\begin{cases}
k^{\frac{N}{2}-\beta} \quad &\text{if} \,\,\,k\ge 2,\\
1 &\text{if} \,\,\,k=1, 0, -1,\\
0 &\text{if} \,\,\,k\le -2.
\end{cases}
\end{align*}
Note that we used the second order Poincar\'e inequality: $C \| u\|_q \le \| \lap u \|_q$ to show (\ref{R_N_k}) in the case where $k \le 1$, see e.g. \cite{GGS}.
Then we have
\begin{align*}
C \sum_{k \in \Z} \int_{B_1} \frac{|u \phi_k|^{\frac{N}{2}}}{|x|^N \( \log \frac{a}{|x|} \)^\beta} \,dx \le \sum_{k \in \Z} b_k \int_{A_k} | \lap (u \phi_k ) |^{\frac{N}{2}} dx
\end{align*}
which yields that
\begin{align}\label{R_N_k sum}
C \int_{B_1} \frac{|u |^{\frac{N}{2}}}{|x|^N \( \log \frac{a}{|x|} \)^\beta} \,dx 
&\le C \sum_{k=2}^\infty b_k \int_{A_k} |\lap \phi_k|^{\frac{N}{2}} | u |^{\frac{N}{2}} + \phi_k^{\frac{N}{2}} |\lap u|^{\frac{N}{2}} + |\nabla u|^{\frac{N}{2}} |\nabla \phi_k |^{\frac{N}{2}} dx \notag \\
&=: C \sum_{k=2}^\infty ( I_1 + I_2 + I_3 ). 
\end{align}
Since $|\lap \phi_k (x)| \le C e^{2(k+1)}$ for $x \in A_k$, by (\ref{estimate u}) in Lemma \ref{2nd-order radial lemma} we have
\begin{align*}
I_1 &\le  C k^{\frac{N}{2} - \beta} e^{N(k+1)} \int_{A_k} |u|^{\frac{N}{2}} \, dx \\ 
&\le C k^{\frac{N}{2} - \beta} e^{N(k+1)} \| \lap u \|_{\frac{N}{2}}^{\frac{N}{2}} \int_{A_k} \( \log \frac{1}{|x|} \)^{\frac{N}{2}} \, dx \\
&\le C k^{\frac{N}{2} - \beta} e^{N(k+1)} \| \lap u \|_{\frac{N}{2}}^{\frac{N}{2}} \int_{k}^{k+2} t^{\frac{N}{2}} e^{-Nt} \, dt \\
&\le C k^{N+1 - \beta} \| \lap u \|_{\frac{N}{2}}^{\frac{N}{2}}.
\end{align*}
In the similar way, we obtain the following estimates of $I_2$ and $I_3$.
\begin{align*}
I_2 + I_3 \le C k^{\frac{N}{2} - \beta} \| \lap u \|_{\frac{N}{2}}^{\frac{N}{2}}.
\end{align*}
Here we used (\ref{estimate u'}) in Lemma \ref{2nd-order radial lemma} to show the estimate of $I_3$. 
From (\ref{R_N_k sum}) and the estimates of $I_i\, (i=1, 2, 3)$ we have
\begin{align*}
C \int_{B_1} \frac{|u |^{\frac{N}{2}}}{|x|^N \( \log \frac{a}{|x|} \)^\beta} \,dx 
\le C \( \sum_{k=2}^{\infty} k^{N+1-\beta} \) \int_{B_1} | \lap u |^{\frac{N}{2}} dx  
\le C \int_{B_1} | \lap u |^{\frac{N}{2}} dx.
\end{align*}
\end{proof}


Let $1 < p < \frac{N}{m}$ and $m \ge 2$. 
The higher order Rellich inequality
\begin{equation*}
C_{m,p}^p \int_{B_R} \frac{|u|^p}{|x|^{mp}} dx \le |u|^p_{m,p} 
\end{equation*}
holds for all $u \in W_0^{m,p}(B_R)$ (see \cite{Rellich}, \cite{Davies-Hinz}, \cite{Mitidieri}).
Here we set
\begin{align*}
&|u|^p_{m,p} = \begin{cases}
              \int_{B_R} | \lap^\ell u |^p \,dx \quad &\text{if} \,\, m=2\ell,\\
              \int_{B_R} | \nabla (\lap^\ell u) |^p \,dx  &\text{if} \,\, m=2\ell +1,
              \end{cases}\\
&C_{m,p} = \begin{cases}
              p^{-2\ell} \prod_{j =1}^{\ell} \{ N-2pj \} \{ N(p-1) +2p(j -1) \}  &\text{if} \, m=2\ell,\\
              \frac{(N-p)}{p^{2(\ell+1)}} \prod_{j =1}^{\ell} \( N-(2j+1)p \) \left\{ N(p-1) +(2j-1)p \right\} \, &\text{if} \,m=2\ell+1,
              \end{cases}
\end{align*}
for $m, \ell \in \N, \ell \ge 1$.

In the higer order case where $m \ge 3$, it is difficult to show the pointwise estimate corresponding to Lemma \ref{2nd-order radial lemma} by the same method in Lemma \ref{2nd-order radial lemma}.  
Due to the lack of good pointwise estimate for radial functions, our limiting procedure as $p \nearrow \frac{N}{m}$ can not work well in the higher order case. 
However, we can show at least the following pointwise estimates for radial functions in $W_0^{m, p}(B_R)$ for $m \ge 2$ via iteration method. 
The following pointwise estimates are not optimal. 
We expect that the pointwise estimates in Proposition \ref{higher order radial lemma} will be applicable somewhere.

\begin{prop}\label{higher order radial lemma}
Let $N, m \ge 3, p \in [1, \frac{N}{2})$ if $m$ is even and $p \in [1, N)$ if $m$ is odd, $u \in C_c^m (B_R)$ be a radial function, and $C$ be a constant which is independent of $u$. Then the following pointwise estimates hold for any $r \in (0, R)$.
\begin{align}\label{pointwise est}
|u(r)| \le C | u|_{m, p} r^{2-N}.
\end{align}
\end{prop}

\begin{proof}
We shall show (\ref{pointwise est})  for $p \in [1, N)$ and odd number $m$ inductively. First we show the case where $m=3$. 
By the transformation (\ref{trans}) for radial function $u$ and the pointwise estimate for radial function $v:= \lap u \in W_0^{1, p}$, we obtain 
\begin{align*}
|v(r)| \le C \| \nabla v \|_p r^{-\frac{N-p}{p}} = C \| \nabla \lap u\|_p (t+1)^{\frac{N-p}{(N-2)p}}.
\end{align*}
By (\ref{lap to ''}) we have
\begin{align*}
|w''(t)| \le C \| \nabla \lap u\|_p (t+1)^a,
\end{align*}
where $a = \frac{N - (2N -1)p}{(N-2)p} < -1$. Therefore we have
\begin{align*}
|w(t)| &\le \int_0^t \int_0^s |w'' (u)| \,du \,ds \\
&\le C \| \nabla \lap u\|_p \int_0^t \int_0^s (u+1)^a \,du \,ds \\
&\le C \| \nabla \lap u\|_p 
{\rm max} \{ (t+1)^{a+2}, \,t+1 \}  \le C \| \nabla \lap u\|_p (t+1).
\end{align*}
Thus we obtain (\ref{pointwise est}) for $m=3$. Next we assume that (\ref{pointwise est}) holds for $m= 2\ell +1$. 
And we shall show that (\ref{pointwise est}) also holds for $m=2(\ell + 1) +1$. 
For a radial function $u \in C_c^{2 \ell +3}$, set $v:= \lap u \in C_c^{2\ell +1}$. Applying (\ref{pointwise est}) for $v$, we have 
\begin{align*}
|v(r)| \le C \| \nabla \lap^\ell v\|_{L^p(B_R)} r^{2-N}.
\end{align*}
By (\ref{trans}) and (\ref{lap to ''}), we have
\begin{align*}
|w''(t)| \le C \| \nabla \lap^{\ell + 1} u\|_{L^p(B_R)} (t+1)^{b},
\end{align*}
where $b= -\frac{2N}{N-2} < -1$. 
Therefore we have
\begin{align*}
|w(t)| &\le \int_0^t \int_0^s |w'' (u)| \,du \,ds \\
&\le C \| \nabla \lap^{\ell +1} u\|_p \int_0^t \int_0^s (u+1)^b \,du \,ds \\
&\le C \| \nabla \lap^{\ell +1} u\|_p 
{\rm max} \{ (t+1)^{b+2}, \,t+1 \}  \le C \| \nabla \lap^{\ell +1} u\|_p (t+1).
\end{align*}
Therefore we observe that (\ref{pointwise est}) holds for $m=2(\ell +1) +1$. 

In the even case, the strategy of the proof is same as the odd case. 
In order to obtain (\ref{pointwise est}) for $m=4$, we use the pointwise estimate in Lemma \ref{2nd-order radial lemma} for radial function $v := \lap u \in C_c^{2}$. 
We omit the proof.
\end{proof}

%
%

\section{A limiting procedure for the Poincar\'e inequality}\label{Poincare}

In this section, we apply our limiting procedure to the Poincar\'e inequality:
\begin{align}\label{Po}
\la (\Omega) \int_\Omega |u|^p \, dx \le \int_{\Omega} |\nabla u|^p \,dx \quad (\,u \in C_c^1(\Omega), 1 \le p < \infty\,).
\end{align} 
The Poincar\'e inequality (\ref{Po}) does not have a critical exponent with respect to $p$ like the Hardy type inequalities.
However the optimal constant $\la (\Omega)$ goes to $\infty$ and $\int_\Omega |u|^p dx$ goes to $0$, as $| \Omega | \searrow 0$. This can be regarded as a kind of limiting situation. 
Recall that 
\begin{align}\label{order Po}
\la (\Omega) \ge \( \frac{N}{p} \,|B_1 | \,\)^p | \Omega |^{-\frac{p}{N}} 
\end{align}
see e.g. \cite{KF}. By using this growth order of $\la (\Omega)$ as $|\Omega | \searrow 0$ and our limiting procedure, we shall consider {\it a limit} for the Poincar\'e inequality as $|\Omega | \searrow 0$.

\begin{theorem}\label{H_p from P}
Let $1 \le p < \frac{N^2}{N-1}$. The following non-sharp classical Hardy inequality (\ref{H_p non-sharp}) can be derived by a limiting procedure for the Poincar\'e inequality (\ref{Po}) as $|\Omega | \searrow 0$. 
\begin{align}\label{H_p non-sharp}
C \int_{B_1} \frac{|u|^p}{|x|^\beta } \,dx \le \int_{B_1} | \nabla u |^p \,dx  \quad (u \in C_{c}^1 (B_1)).
\end{align}
Here the constant $C= C(\beta, p, N)>0$ is independent of $u$ and $\beta > 0$ satisfies
\begin{align*}
\begin{cases}
\beta < \frac{p}{N} \quad &\text{if} \,\, 1 \le p \le N,\\
\beta < \frac{p}{N} +N -p &\text{if} \,\, N< p < \frac{N^2}{N-1}.
\end{cases}
\end{align*}
\end{theorem}

\begin{remark}
If $p \ge \frac{N^2}{N-1}$, then we can not obtain any information which is better than the Poincar\'e inequality (\ref{Po}) by out limiting procedure as $|\Omega | \searrow 0$, since $\beta =0$ in that case.
\end{remark}

\begin{proof}
From Lemma \ref{radial nomi}, it is enough to show the inequality (\ref{H_p non-sharp}) for any radial functions $u \in C_c^{1}(B_1)$. Let $1 \le p < N$ and $\{ \phi_k \}_{k \in \Z} \subset  C_c^{\infty}(\re^N \setminus \{ 0\})$ be radial functions which satisfy
\begin{align*}
&(i) \,\sum_{k=-\infty}^{+\infty} \phi_k (x)^p =1, \,0 \le \phi_k (x) \le 1 \,\, \( \forall x \in \re^N \setminus \{ 0\} \), \\
&(ii) \,\,\text{supp}\,\phi_k \subset B_{1/k} \setminus B_{1/k+2}.
\end{align*}
Set $u_k = u \,\phi_k$ and $A_k=$ supp\,$u_k \subset B_1 \cap \( B_{1/k} \setminus B_{1/k+2}\)$. 
Applying the Poincar\'e inequality (\ref{Po}) for $u_k$ and (\ref{order Po}), we have 
\begin{align*}
C k^p (k+2)^{\frac{p}{N}} \int_{A_k} |u_k|^p\, dx \le \int_{A_k} | \nabla u_k |^{p} \,dx. 
\end{align*}
Since $k \le \frac{1}{|x|} \le k+2$ for $x \in A_k$, 
\begin{align}\label{P_k}
C \int_{A_k} \frac{|u_k|^p}{|x|^\beta} \,dx \le b_k \int_{A_k} | \nabla u_k |^p dx
\end{align}
for $k \in \Z$, and $\beta > 2N$, where $b_k$ is given by
\begin{align*}
b_k=
\begin{cases}
k^{-p} (k+2)^{\beta - \frac{p}{N}} \quad &\text{if} \,\,\,k\ge 1,\\
1 &\text{if} \,\,\,k=0, -1,\\
0 &\text{if} \,\,\,k\le -2.
\end{cases}
\end{align*}
Summing both sides on (\ref{P_k}), we have
\begin{align*}
C \sum_{k \in \Z} \int_{B_1} \frac{|u \phi_k|^p}{|x|^\beta} \,dx \le \sum_{k \in \Z} b_k \int_{A_k} | \nabla (u \phi_k ) |^N dx.
\end{align*}
By applying Lemma \ref{radial lemma} and caluculating in the similar way to it in \S \ref{limiting procedure}, we see that for $\beta < \frac{p}{N}$ the desired inequality (\ref{H_p non-sharp}) can be obtained. 
In the case where $N \le p < \frac{N^2}{N-1}$, the proof is similar. 
Therefore we omit the proof in that case. 
\end{proof}


\section*{Acknowledgment}
This work was supported by the Research Institute for Mathematical
Sciences, an International Joint Usage/Research Center located in Kyoto
University and was (partly) supported by Osaka City University Advanced
Mathematical Institute (MEXT Joint Usage/Research Center on Mathematics
and Theoretical Physics). 
And also, the first author was supported by JSPS KAKENHI Early-Career Scientists, No. JP19K14568 and the second author was partially supported by JSPS KAKENHI Grant-in-Aid for Scientific Research(B), No. JP15H03621.

The authors thank Prof. Yuki Naito (Ehime University) for giving them a useful comment.


\end{document}